\documentclass[12pt]{article}
\usepackage[centertags]{amsmath}
\usepackage{amsfonts}
\usepackage{amssymb}
\usepackage{latexsym}
\usepackage{amsthm}
\usepackage{newlfont}
\usepackage{graphicx}
\usepackage{listings}
\usepackage{booktabs}
\usepackage{abstract}
\date{}

\newlength{\defbaselineskip}
\newcommand{\setlinespacing}[1]%
           {\setlength{\baselineskip}{#1 \defbaselineskip}}

\newcommand{\N}{{\mathbb{N}}}

\newcommand{\actaqed}{\hfill $\actabox$}
{\medskip\noindent \textit{Proof of #1. }}%
{\actaqed \medskip}

\def\cB{\mathcal B}
\def\cC{{\mathcal C}}

\def \Tr{\mathcal T}
\def \K{\mathcal K}

\def \V{\mathcal V}

\def \cX{\mathcal X}

\def \cM{\mathcal M}
\def\R{{\mathbb R}}
\def\Z{\mathbb Z}

\def \T{\mathbb T}

\def \<{\langle}
\def\>{\rangle}

\def\ba{\mathbf a}
\def\bx{\mathbf x}
\def\by{\mathbf y}

\def\bk{\mathbf k}
\def\bu{\mathbf u}
\def\bv{\mathbf v}
\def\bw{\mathbf w}

\def\bn{\mathbf n}
\def\bs{\mathbf s}
\def\br{\mathbf r}
\def\bbe{\mathbf e}
\def\bN{\mathbf N}

\def\bK{\mathbf K}

\newtheorem{Theorem}{Theorem}[section]
\newtheorem{Lemma}{Lemma}[section]
\newtheorem{Definition}{Definition}[section]
\newtheorem{Proposition}{Proposition}[section]

\numberwithin{equation}{section}

\newcommand{\be}{\begin{equation}}
\newcommand{\ee}{\end{equation}}

\begin{document}

\title{Universal discretization}
\author{V.N. Temlyakov\thanks{University of South Carolina and Steklov Institute of Mathematics.  }}
\maketitle
\begin{abstract}
{The paper is devoted to discretization of integral norms of functions from
a given collection of finite dimensional subspaces. For natural collections of subspaces of the multivariate trigonometric polynomials we construct sets of points, which are optimally (in the sense of order) good for each subspace of a collection from the point of view of the integral norm discretization. We call such sets universal. Our construction of the universal sets is based on deep results on 
 existence of special nets, known as $(t,r,d)$-nets. }
\end{abstract}

\section{Introduction} 
\label{I} 

This paper is a follow up to the two recent papers \cite{VT160} and \cite{VT161}. As it is clear from the title the two main concepts of the paper are {\it discretization} and {\it universality}.  Discretization is 
a very important step in making a continuous problem computationally feasible. The problem of construction of good sets of points in a multidimensional domain is a fundamental problem of mathematics and computational mathematics. 
A prominent  example of classical discretization problem is a problem of metric entropy (covering numbers, entropy numbers). 
Bounds for the entropy numbers of function classes are important by themselves and also have important connections to other fundamental problems (see, for instance, \cite{Tbook}, Ch.3 and \cite{DTU}, Ch.6). 
Another prominent example of a discretization problem is the problem of numerical integration.  Numerical integration in the mixed smoothness classes requires deep number theoretical results for constructing optimal (in the sense of order) cubature formulas (see, for instance, \cite{DTU}, Ch.8). A typical approach to solving a continuous problem numerically -- the Galerkin method -- 
suggests to look for an approximate solution from a given finite dimensional subspace. A standard way to measure an error of approximation is an appropriate $L_q$ norm, $1\le q\le\infty$. Thus, the problem of   discretization of the $L_q$ norms of functions from a given finite dimensional subspace arises in a very natural way.

We now discuss the universality principle formulated in \cite{Tem16}.  The
concept of smoothness becomes more complicated in the multivariate case than it
is in the univariate case. In the multivariate case a function may have
different smoothness properties in different coordinate directions. In other
words, functions may belong to different anisotropic smoothness classes (see anisotropic Sobolev and 
 Nikol'skii 
classes $W^\br_{q,\alpha}$ and $H^\br_q$ in \cite{TBook}, Chapter 2). It is known   that approximation characteristics of
anisotropic smoothness classes depend on the average smoothness $g(\br)$ and optimal
approximation methods depend on anisotropy of classes, on the vector $\br$. This motivated a study in
\cite{Tem16} of existence of an approximation method that is good for all
anisotropic smoothness classes. This is a problem of existence of a universal
method of approximation.  There is a similar setting for existence of universal cubature formulas (see, for instance, \cite{TBook}, Chapter 4).
We note that the universality concept in learning
theory is very important and it is close to the concepts of adaptation and
distribution-free estimation in non-parametric statistics (\cite{GKKW},
\cite{BCDDT}, \cite{VT113}). 

We now proceed to the detailed presentation. 

{\bf Marcinkiewicz problem.} Let $\Omega$ be a compact subset of $\R^d$ with the probability measure $\mu$. We say that a linear subspace $X_N$ of the $L_q(\Omega)$, $1\le q < \infty$, admits the Marcinkiewicz-type discretization theorem with parameters $m$ and $q$ if there exist a set $\{\xi^\nu \in \Omega, \nu=1,\dots,m\}$ and two positive constants $C_j(d,q)$, $j=1,2$, such that for any $f\in X_N$ we have
\be\label{1.1}
C_1(d,q)\|f\|_q^q \le \frac{1}{m} \sum_{\nu=1}^m |f(\xi^\nu)|^q \le C_2(d,q)\|f\|_q^q.
\ee
In the case $q=\infty$ we define $L_\infty$ as the space of continuous on $\Omega$ functions and ask for 
\be\label{1.2}
C_1(d)\|f\|_\infty \le \max_{1\le\nu\le m} |f(\xi^\nu)| \le  \|f\|_\infty.
\ee
We will also use a brief way to express the above property: the $\cM(m,q)$ theorem holds for  a subspace $X_N$ or $X_N \in \cM(m,q)$. 

{\bf Universal discretization problem.} This problem is about finding (proving existence) of 
a set of points, which is good in the sense of the above Marcinkiewicz-type discretization 
for a collection of linear subspaces. We formulate it in an explicit form. Let $\cX_N:= \{X_N^j\}_{j=1}^k$ be a collection of linear subspaces $X_N^j$ of the $L_q(\Omega)$, $1\le q \le \infty$. We say that a set $\{\xi^\nu \in \Omega, \nu=1,\dots,m\}$ provides {\it universal discretization} for the collection $\cX_N$ if, in the case $1\le q<\infty$, there are two positive constants $C_i(d,q)$, $i=1,2$, such that for each $j\in [1,k]$ and any $f\in X_N^j$ we have
\be\label{1.1u}
C_1(d,q)\|f\|_q^q \le \frac{1}{m} \sum_{\nu=1}^m |f(\xi^\nu)|^q \le C_2(d,q)\|f\|_q^q.
\ee
In the case $q=\infty$  for each $j\in [1,k]$ and any $f\in X_N^j$ we have
\be\label{1.2u}
C_1(d)\|f\|_\infty \le \max_{1\le\nu\le m} |f(\xi^\nu)| \le  \|f\|_\infty.
\ee

We study the  Universal discretization for subspaces of the 
trigonometric polynomials. Let $Q$ be a finite subset of $\Z^d$. We denote
$$
\Tr(Q):= \{f: f=\sum_{\bk\in Q}c_\bk e^{i(\bk,\bx)}\}.
$$
We briefly present well known results related to the Marcinkiewicz-type discretization theorems for the trigonometric polynomials. We begin with the case $Q=\Pi(\bN):=[-N_1,N_1]\times \cdots \times [-N_d,N_d]$, $N_j \in \N$ or $N_j=0$, $j=1,\dots,d$, $\bN=(N_1,\dots,N_d)$. 
We denote
\begin{align*}
P(\mathbf N) := \bigl\{\mathbf n = (n_1 ,\dots,n_d),&\qquad n_j\ -\
\text{ are nonnegative integers},\\
&0\le n_j\le 2N_j  ,\qquad j = 1,\dots,d \bigr\},
\end{align*}
and set
$$
\bx^{\mathbf n}:=\left(\frac{2\pi n_1}{2N_1+1},\dots,\frac{2\pi n_d}
{2N_d+1}\right),\qquad \mathbf n\in P(\mathbf N) .
$$
Then for any $t\in \Tr(\Pi(\mathbf N))$
\be\label{b1.5}
\|t\|_2^2  =\vartheta(\mathbf N)^{-1}\sum_{\mathbf n\in P(\mathbf N)}
\bigl|t(\bx^{\mathbf n})\bigr|^2   
\ee
where $\vartheta(\mathbf N) := \prod_{j=1}^d (2N_j  + 1)=\dim\Tr(\Pi(\bN))$.  
The following version of (\ref{b1.5}) for $1<q<\infty$ is the well known Marcinkiewicz discretization theorem (for $d=1$) (see \cite{Z}, Ch.10, \S7 and \cite{TBook}, Ch.1, Section 2)
$$
C_1(d,q)\|t\|_q^q  \le\vartheta(\mathbf N)^{-1}\sum_{\mathbf n\in P(\mathbf N)}
\bigl|t(\bx^{\mathbf n})\bigr|^q  \le C_2(d,q)\|t\|_q^q,\quad 1<q<\infty,
$$
which implies  
\be\label{1.4}
\Tr(\Pi(\bN)) \in \cM(\vartheta(\bN),q),\quad 1<q<\infty.
\ee
Some modifications are needed in the case $q=1$ or $q=\infty$. 
Denote
\begin{align*}
P'(\mathbf N) := \bigl\{\mathbf n &= (n_1,\dots,n_d),\qquad n_j\ -\
\text{ are natural numbers},\\
&0\le n_j\le 4N_j -1 ,\qquad j = 1,\dots,d \bigr\}
\end{align*}
and set
$$
\bx(\mathbf n) :=\left (\frac{\pi n_1}{2N_1} ,\dots,\frac{\pi n_d}{2N_d}
\right)   ,\qquad \mathbf n\in P'(\mathbf N)  .
$$
In the case $N_j  = 0$ we assume $x_j (\mathbf n) = 0$. Denote ${\overline N} := \max (N,1)$ and $v(\bN) := \prod_{j=1}^d {\overline N_j}$. Then the following Marcinkiewicz-type discretization theorem is known for all $1\le q\le \infty$
\be\label{1.4'}
C_1(d,q)\|t\|_q^q  \le v(4\mathbf N)^{-1}\sum_{\mathbf n\in P'(\mathbf N)}
\bigl|t(\bx({\mathbf n}))\bigr|^q  \le C_2(d,q)\|t\|_q^q,
\ee
which implies the following relation
\be\label{1.5}
\Tr(\Pi(\bN)) \in \cM(v(4\bN),q),\quad 1\le q\le \infty.
\ee
Note that $v(4\bN) \le C(d) \dim \Tr(\Pi(\bN))$. 

It is clear from the above construction that 
the set $\{\bx(\bn): \bn\in P'(\bN)$ depends substantially on $\bN$. The main goal of this paper is to construct for a given $q$ and $M$ a set, which satisfies an analog of (\ref{1.4'}) for all $\bN$ with $v(\bN)\le M$. 
In this paper we are primarily interested in the Universal discretization  for the collection of subspaces of trigonometric polynomials with frequencies from parallelepipeds (rectangles). For $\bs\in\Z^d_+$
define
$$
R(\bs) := \{\bk \in \Z^d :   |k_j| < 2^{s_j}, \quad j=1,\dots,d\}.
$$
Clearly, $R(\bs) = \Pi(\bN)$ with $N_j = 2^{s_j}-1$. Consider the collection $\cC(n,d):= \{\Tr(R(\bs)), \|\bs\|_1=n\}$. 

We prove here the following result.
\begin{Theorem}\label{T1.1}   For every $1\le q\le\infty$ there exists a large enough constant $C(d,q)$, which depends only on $d$ and $q$, such that for any $n\in \N$ there is a set $\Xi_m:=\{\xi^\nu\}_{\nu=1}^m\subset \T^d$, with $m\le C(d,q)2^n$ that provides universal discretization in $L_q$   for the collection $\cC(n,d)$.
\end{Theorem}
Theorem \ref{T1.1}, basically, solves the Universal discretization problem for the collection $\cC(n,d)$. It provides the upper bound  $m\le C(d,q)2^n$ with $2^n$ being of the order of the dimension of each $\Tr(R(\bs))$ from the collection $\cC(n,d)$.
 Obviously, the lower bound for the cardinality of a set, providing the Marcinkiewicz discretization theorem for $\Tr(R(\bs))$ with $\|\bs\|_1=n$, is $\ge C(d)2^n$. 
We treat separately the case $q=\infty$ (see Section \ref{Inf}) and the case $1\le q<\infty$ (see Section \ref{Lq}). Our construction of the universal set is based on deep results on 
 existence of special nets, known as $(t,r,d)$-nets (see Section \ref{Inf} for the definition).

\section{Universal discretization in $L_\infty$} 
\label{Inf}

We begin with a simple observation. Denote for brevity $\Tr(\bN) := \Tr(\Pi(\bN))$, $\T^d := [0,2\pi)^d$. 

\begin{Proposition}\label{P2.1} Let a set $(\xi^1,\dots,\xi^m) \subset \T^d$ be such that for any $\bx \in \T^d$ there exists $\xi^{\nu(\bx)}$ such that
\be\label{2.1}
|x_j-\xi_j^{\nu(\bx)}| \le (2dN_j)^{-1}, \quad j=1,\dots,d.
\ee
Then for each $f\in \Tr(\bN)$, $\bN \in \N^d$, we have 
$$
\|f\|_\infty \le 2\max_{1\le \nu \le m} |f(\xi^\nu)|.
$$
\end{Proposition}
\begin{proof} The Bernstein inequality for $f\in\Tr(\bN)$ 
$$
\|f'_{x_j}\|_\infty \le N_j\|f\|_\infty
$$
shows, by the Mean Value Theorem (the Lagrange theorem), that for any $\by\in \T^d$ we have for all coordinate directions $\bbe_j$ and all $h$
$$
|f(\by+h\bbe_j) -f(\by)| \le N_j|h|\|f\|_\infty.
$$
This easily implies that for any $\bx,\by \in \T^d$ 
\be\label{2.2}
|f(\bx)-f(\by)| \le \sum_{j=1}^d N_j|x_j-y_j|\|f\|_\infty.
\ee
Let $\bx^0$ be a point of maximum of $|f(\bx)|$. Then (\ref{2.1}) and (\ref{2.2}) give
$$
|f(\bx^0)-f(\xi^{\nu(\bx^0)})| \le \frac{1}{2}\|f\|_\infty
$$
and, therefore, 
$$
\max_{1\le \nu\le m}|f(\xi^\nu)| \ge |f(\xi^{\nu(\bx^0)})| \ge \frac{1}{2}\|f\|_\infty.
$$
\end{proof}

We now show how known results on the {\it dispersion} of a finite set of points can be 
used in the Universal discretization problem in $L_\infty$. We remind the definition of 
dispersion. Let $d\ge 2$ and $[0,1)^d$ be the $d$-dimensional unit cube. For $\bx,\by \in [0,1)^d$ with $\bx=(x_1,\dots,x_d)$ and $\by=(y_1,\dots,y_d)$ we write $\bx < \by$ if this inequality holds coordinate-wise. For $\bx<\by$ we write $[\bx,\by)$ for the axis-parallel box $[x_1,y_1)\times\cdots\times[x_d,y_d)$ and define
$$
\cB:= \{[\bx,\by): \bx,\by\in [0,1)^d, \bx<\by\}.
$$
For $n\ge 1$ let $T$ be a set of points in $[0,1)^d$ of cardinality $|T|=n$. The volume of the largest empty (from points of $T$) axis-parallel box, which can be inscribed in $[0,1)^d$, is called the dispersion of $T$:
$$
\text{disp}(T) := \sup_{B\in\cB: B\cap T =\emptyset} vol(B).
$$
An interesting extremal problem is to find (estimate) the minimal dispersion of point sets of fixed cardinality:
$$
\text{disp*}(n,d) := \inf_{T\subset [0,1)^d, |T|=n} \text{disp}(T).
$$
It is known that 
\be\label{2.3}
\text{disp*}(n,d) \le C^*(d)/n.
\ee
Inequality (\ref{2.3}) with $C^*(d)=2^{d-1}\prod_{i=1}^{d-1}p_i$, where $p_i$ denotes the $i$th prime number, was proved in \cite{DJ} (see also \cite{RT}). The authors of \cite{DJ} used the Halton-Hammersly set of $n$ points (see \cite{Mat}). Inequality (\ref{2.3}) with $C^*(d)=2^{7d+1}$ was proved in 
\cite{AHR}. The authors of \cite{AHR}, following G. Larcher, used the $(t,r,d)$-nets. 
\begin{Definition}\label{D2.1} A $(t,r,d)$-net (in base $2$) is a set $T$ of $2^r$ points in 
$[0,1)^d$ such that each dyadic box $[(a_1-1)2^{-s_1},a_12^{-s_1})\times\cdots\times[(a_d-1)2^{-s_d},a_d2^{-s_d})$, $1\le a_j\le 2^{s_j}$, $j=1,\dots,d$, of volume $2^{t-r}$ contains exactly $2^t$ points of $T$.
\end{Definition}

\begin{Theorem}\label{T2.1}   There exists a large enough constant $C(d)$, which depends only on $d$, such that for any $n\in \N$ there is a set $\Xi_m:=\{\xi^\nu\}_{\nu=1}^m\subset \T^d$, with $m\le C(d)2^n$ that provides universal discretization in $L_\infty$ for the collection $\cC(n,d)$.
\end{Theorem}
\begin{proof} We will use a $(t,r,d)$-net, which satisfies the conditions of Proposition \ref{P2.1} for all $\bN = 2^\bs$, $\|\bs\|_1=n$. More precisely, it will be a net adjusted to the 
$\T^d$, namely, a net $T$ for $[0,1)^d$ multiplied by $2\pi$.   A construction of such nets for all $d$ and $t\ge Cd$, $r\ge t$  is given in \cite{NX}. For a given $n\in \N$ define $r:= n+t+(1+[\log(4d\pi)])d$ and consider the $(t,r,d)$-net $T(t,r,d)$. Denote $T_\pi(t,r,d):= 2\pi T(t,r,d):=\{2\pi \bx: \bx\in T(t,r,d)\}$. Take an $\bs\in\Z^d_+$ such that $\|\bs\|_1=n$ and consider $\bs'$ with $s_j' := s_j+1+[\log (4d\pi)]$. Any $\bx\in \T^d$ belongs to a dyadic box of the form $[2\pi(a_1-1)2^{-s'_1},2\pi a_12^{-s'_1})\times\cdots\times[2\pi(a_d-1)2^{-s'_d},2\pi a_d2^{-s'_d})$ with proper $a_1,\dots,a_d$. The volume of this box is equal to
$(2\pi)^d2^{t-r}$. Therefore, by the definition of the $T_\pi(t,r,d):=\{\xi^\nu\}_{\nu=1}^{2^r}$ there is 
$\xi^{\nu(\bx)}$ from that box. Note that here we use a weaker property of $T(t,r,d)$ than required by Definition \ref{D2.1} -- we only need one point from the set to be in the box. Then for each $j=1,\dots,d$
$$
|x_j-\xi^{\nu(\bx)}_j| \le 2\pi 2^{-s'_j} \le 2^{-s_j} 2\pi(4d\pi)^{-1} = (2d2^{s_j})^{-1}.
$$
This means that $\{\xi^\nu\}_{\nu=1}^{2^r}$ satisfies the condition of Proposition \ref{P2.1} with $m=2^r \le C(d)2^n$ for all $\bs\in\Z^d_+$ with $\|\bs\|_1=n$. By Proposition \ref{P2.1} the set $\{\xi^\nu\}_{\nu=1}^{2^r}$
provides for each $f\in \Tr(2^\bs)$ the inequality
$$
\|f\|_\infty \le 2\max_{1\le \nu \le m} |f(\xi^\nu)|.
$$
This completes the proof of Theorem \ref{T2.1}.

\end{proof} 

In the above theorem we used a set of points with special structure -- the $(t,r,d)$-net -- 
as the one, which provided the universal discretization for the collection $\cC(n,d)$. 
We now prove a conditional theorem based on the concept of dispersion. 

\begin{Theorem}\label{T2.2} Let a set $T$ with cardinality $|T|= 2^r=:m$ have dispersion 
satisfying the bound disp$(T) < C(d)2^{-r}$ with some constant $C(d)$. Then there exists 
a constant $c(d)\in \N$ such that the set $2\pi T:=\{2\pi\bx: \bx\in T\}$ provides the universal discretization in $L_\infty$ for the collection $\cC(n,d)$ with $n=r-c(d)$.
\end{Theorem}
\begin{proof} The proof is based on Proposition \ref{P2.1}. The proof goes along the lines of the above proof of Theorem \ref{T2.1}. We use notations from the proof of Theorem \ref{T2.1}. Set
$$
c(d) := 1+[\log C(d)] + d(1+ [\log 4d\pi]).
$$
Take an $\bs\in\Z^d_+$ such that $\|\bs\|_1=n=r-c(d)$ (assuming $r>c(d)$) and consider $\bs'$ as in the proof of Theorem \ref{T2.1}.   Any $\bx\in \T^d$ belongs to a dyadic box of the form $[2\pi(a_1-1)2^{-s'_1},2\pi a_12^{-s'_1})\times\cdots\times[2\pi(a_d-1)2^{-s'_d},2\pi a_d2^{-s'_d})$ with some $a_1,\dots,a_d$. The volume of this box is equal to
$(2\pi)^d2^{-\|\bs'\|_1} \ge (2\pi)^dC(d)2^{-r}$. Therefore, by the definition of dispersion this box has a point from $2\pi T$. Denote $2\pi T :=\{\xi^\nu\}_{\nu=1}^{2^r}$ and denote a point from the the above box by $\xi^{\nu(\bx)}$. Then for each $j=1,\dots,d$
$$
|x_j-\xi^{\nu(\bx)}_j| \le 2\pi 2^{-s'_j} \le 2^{-s_j} 2\pi(4d\pi)^{-1} = (2d2^{s_j})^{-1}.
$$
We complete the proof in the same way as in the proof of Theorem \ref{T2.1}.
 \end{proof}
 
 It turns out that two properties of a set $T$: (I) dispersion of $T$ is of order $2^{-n}$ and 
 (II) the set $2\pi T$ provides universal discretization in $L_\infty$ for the collection 
 $\cC(n,d)$ are closely connected. Theorem \ref{T2.2} shows that, loosely speaking, 
 property (I) implies property (II). We now prove the inverse implication. 
 
 \begin{Theorem}\label{T2.2'}  Assume that $T\subset [0,1)^d$ is such that the set $2\pi T$ provides universal discretization in $L_\infty$ for the collection 
 $\cC(n,d)$ with a constant $C_1(d)$ (see (\ref{1.2u})). Then there exists a positive constant $C(d)$ with the following property disp$(T) \le C(d)2^{-n}$. 
 \end{Theorem}
 \begin{proof} We prove this theorem by contradiction. Suppose that disp$(T) > C(d)2^{-n}$. We will show that for large enough $C(d)$ this contradicts (\ref{1.2u})
 for the collection $\cC(n,d)$. Inequality disp$(T) > C(d)2^{-n}$ with large enough $C(d)$ implies that there exists $\bs\in \Z^d_+$, $\|\bs\|_1 =n$, and an empty box $[\bu,\bv)\subset [1,0)^d$ such that $|v_j-u_j|\ge 2^{a(d)-s_j}$, $j=1,\dots,d$. Note that the bigger the $C(d)$ the bigger $a(d)$ can be chosen. We will specify it later. Denote $\bw:=(\bu+\bv)/2$ the center of the box $[\bu,\bv)$ and consider (see Section \ref{Lq} for the definition and properties of $\K_\bN$)
 $$
 f(\bx) := \K_{2^\bs}(\bx-\bw) \in \Tr(R(\bs)).
 $$
 Then by (\ref{FKm}) $f(\bw)= 2^{\|\bs\|_1} =2^n$ and by (\ref{a1.10}) outside the box $2\pi [\bu,\bv)$ we have 
 $|f(\bx)| \le 2^{-2a(d)d}f(\bw)$. Thus, making $a(d)$ large enough to satisfy $2^{-2a(d)d}<C_1(d)$ we obtain a contradiction. 
 
 \end{proof}
 
 We now discuss a straight forward  way of constructing point sets for universal discretization. The following statement is obvious.
 \begin{Proposition}\label{P2.2} Suppose that for each subspace $X^i$, $i=1,\dots,k$, there 
 exists a set $\{\xi^\nu(i)\}_{\nu=1}^{m_i}$ such that for any $f\in X^i$ we have
 $$
 \|f\|_\infty \le C\max_{1\le\nu \le m_i} |f(\xi^\nu(i))|
 $$
 with a constant $C$ independent of $i$. Then for any $f\in \cup_{i=1}^k X^i$ we have
 $$
 \|f\|_\infty \le C\max_{1\le i\le k}\max_{1\le\nu \le m_i} |f(\xi^\nu(i))|.
 $$
\end{Proposition}
Proposition \ref{P2.2} shows that the set $\cup_{i=1}^k\{\xi^\nu(i)\}_{\nu=1}^{m_i}$ is a universal set for discretization of the collection $X^1,\dots,X^k$ in the uniform norm. 
Note that the cardinality of this set is $\le m_1+\cdots +m_k$. Let us discuss an application of this idea to the collection $\cC(n,d)$. Take $\bs\in\Z^d_+$ such that $\|\bs\|_1=n$ and choose 
  the set $\{\bx(\bn), \bn \in P'(2^\bs)\}$ (see Introduction), which is a discretization set for $\Tr(2^\bs)$ in $L_\infty$ (see (\ref{1.4'})). The union 
  $$
  SG(n,d) := \cup_{\bs:\|\bs\|_1=n} \{\bx(\bn), \bn \in P'(2^\bs)\}
  $$
  is known as the {\it sparse grids}. 
  
  Proposition \ref{P2.2} shows that the sparse grids $SG(n,d)$ provide  universal discretization for the collection $\cC(n,d)$. It is well known and easy to check that the cardinality of $SG(n,d)$ is of order $2^n n^{d-1}$. Comparing this bound to the bound in Theorem \ref{T2.1} we see that the sparse grids do not provide an optimal (in the sense of order) set for universal discretization. 
  
 Let us explore one more idea of constructing point sets for universal discretization of the collection $\cC(n,d)$. It is clear from the definition of the $R(\bs) = \Pi(\bN)$, with $N_j=2^{s_j}-1$, that 
 $$
 \cup_{\bs,\|\bs\|_1=n}\Tr(R(\bs)) \subset \Tr(Q_n),\qquad Q_n := \cup_{\bs,\|\bs\|_1=n}R(\bs).
 $$
 Certainly, a set $\{\xi^\nu\}_{\nu=1}^m$ which provides discretization for the subspace of 
 the hyperbolic cross polynomials $\Tr(Q_n)$ also provides universal discretization for the collection $\cC(n,d)$. This way of construction of universal methods of approximation and 
 universal cubature formulas works very well (see \cite{TBook}). 
 However, surprisingly, it does not work for building a set, which provides universal discretization for the collection $\cC(n,d)$ in $L_\infty$. This follows from a very nontrivial lower bound for the number of points in a set that provides the Marcinkiewicz discretization theorem for $\Tr(Q_n)$ in $L_\infty$ (see \cite{KT3}, \cite{KT4}, and \cite{KaTe03}). The authors proved that the necessary condition for
$\Tr(Q_n)\in\cM(m,\infty)$ is $m\gg |Q_n|^{1+c}$ with absolute constant $c>0$.
This lower bound and the upper bound from Theorem \ref{T2.1} show that on this way we loose not only an extra $(\log m)^{d-1}$ factor (as in the case of sparse grids) but we loose even an extra factor of the power type $m^c$. 

\section{Universal discretization in $L_q$, $1\le q<\infty$}
\label{Lq}

The main goal of this section is to extend Theorem \ref{T2.1} to the case $1\le q<\infty$, i.e. to prove the following theorem. 

\begin{Theorem}\label{T3.1}   For every $1\le q<\infty$ there exists a large enough constant $C(d,q)$, which depends only on $d$ and $q$, such that for any $n\in \N$ there is a set $\Xi_m:=\{\xi^\nu\}_{\nu=1}^m\subset \T^d$, with $m\le C(d,q)2^n$ that provides universal discretization in $L_q$   for the collection $\cC(n,d)$.
\end{Theorem}
 \begin{proof} We prove this theorem in two steps. First, under some assumptions on the set $\Xi_m$ we prove one-sided discretization inequalities, which bound the discrete norm by the continuous norm. 
 Second, under stronger assumptions on the set $\Xi_m$ we prove the inverse inequalities. 
 
 \subsection{Upper bounds for the discrete norm} \label{Lq1}
 We need some classical trigonometric polynomials. We begin with the univariate case. 
 The Dirichlet kernel of order $n$:
$$
\mathcal D_n (x):= \sum_{|k|\le n}e^{ikx} = e^{-inx} (e^{i(2n+1)x} - 1)
(e^{ix} - 1)^{-1} 
$$
\be\label{a1.1}
=\bigl(\sin (n + 1/2)x\bigr)\bigm/\sin (x/2)
\ee
   is an even trigonometric polynomial.  The Fej\'er kernel of order $n - 1$:
$$
\mathcal K_{n} (x) := n^{-1}\sum_{k=0}^{n-1} \mathcal D_k (x) =
\sum_{|k|\le n} \bigl(1 - |k|/n\bigr) e^{ikx} 
$$
$$
=\bigl(\sin (nx/2)\bigr)^2\bigm /\bigl(n (\sin (x/2)\bigr)^2\bigr).
$$
The Fej\'er kernel is an even nonnegative trigonometric
polynomial in $\Tr(n-1)$ with the majorant
\be\label{a1.10}
\bigl| \mathcal K_{n} (x) \bigr|=\mathcal K_{n} (x)\le \min \bigl(n, \pi^2 /(nx^2)\bigr),
\qquad |x|\le \pi.
\ee
It satisfies the obvious relations
\be\label{FKm}
\| \mathcal K_{n} \|_1 = 1, \qquad \| \mathcal K_{n} \|_{\infty} = n
\ee
The de la Vall\'ee Poussin kernel:
$$
\mathcal V_{n} (x) := n^{-1}\sum_{l=n}^{2n-1} \mathcal D_l (x)= 2\K_{2n}-\K_n, 
$$
is an even trigonometric
polynomials of order $2n - 1$ with the majorant
\be\label{a1.12}
\bigl| \mathcal V_{n} (x) \bigr| \le C \min \bigl(n,  \
 (nx^2)^{-1}\bigr), \quad |x|\le \pi.
\ee
The above relation (\ref{a1.12}) easily implies the following lemma.
\begin{Lemma}\label{L3.1} For a set $\Xi_m:=\{\xi^\nu\}_{\nu=1}^m\subset \T$ satisfying the condition $|\Xi_m\cap [x(l-1),x(l))|\le b$, $x(l) := \pi l/2n$, $l = 1, \dots, 4n$, we have
$$
\sum_{\nu=1}^m\bigl| \mathcal V_n (x - \xi^\nu) \bigr|\le Cbn.
$$
\end{Lemma}
We use the above Lemma \ref{L3.1} to prove a one-sided inequality.

\begin{Lemma}\label{L3.2} For a set $\Xi_m:=\{\xi^\nu\}_{\nu=1}^m\subset \T$ satisfying the condition $|\Xi_m\cap [x(l-1),x(l))|\le b$, $x(l) := \pi l/2n$, $l = 1, \dots, 4n$, we have for $1\le q\le\infty$
$$
\left\|m^{-1}\sum_{\nu=1}^m a_\nu \mathcal V_n (x - \xi^\nu)\right\|_q \le C(bn/m)^{1-1/q} \left(\frac{1}{m}\sum_{\nu=1}^m |a_\nu|^q\right)^{1/q}.
$$
\end{Lemma}
\begin{proof} Denote by $\ell_q(m)$ the $\mathbb C^m$ equipped with the norm
$$
\|\ba\|_q :=\left(\frac{1}{m}\sum_{\nu=1}^m |a_\nu|^q\right)^{1/q},\quad \ba=(a_1,\dots,a_d).
$$
Let $V_n$ be the operator on $\ell_q(m)$ defined as follows:
$$
V(\ba) :=\frac{1}{m}\sum_{\nu=1}^{m} a_\nu \mathcal V_n (x - \xi^\nu).
$$
It is known that $\|\V_n\|_1 \le 3$. This implies
\be\label{a2.22}
\| V_n \|_{\ell_1(m)\to L_1}\le 3.
\ee
Using Lemma \ref{L3.1}
we obtain
\be\label{a2.23}
\| V_n \|_{\ell_{\infty}(m)\to L_{\infty}}\le C bn/m.
\ee
From relations (\ref{a2.22}) and (\ref{a2.23}),  using the Riesz-Torin interpolation 
theorem we find
$$
\| V \|_{\ell_q(m)\to L_q}\le C (bn/m)^{1-1/q},
$$
which implies the conclusion of the lemma.
\end{proof}

We now proceed to the multivariate case. Denote the multivariate Fej\'er and de la Vall\'ee Poussin kernels:
$$
\K_\bN(\bx):=\prod_{j=1}^d \K_{N_j}(x_j), \qquad \mathcal V_{\mathbf N} (\bx) := \prod_{j=1}^d \mathcal V_{N_j}  (x_j)  ,\qquad
\mathbf N = (N_1 ,\dots,N_d) .
$$
In the same way as above in the univariate case one can establish the following multivariate analog of Lemma \ref{L3.2}.

\begin{Lemma}\label{L3.3} For a set $\Xi_m:=\{\xi^\nu\}_{\nu=1}^m\subset \T^d$ satisfying the condition $|\Xi_m\cap [\bx(\bn),\bx(\bn+{\mathbf 1}))|\le b$, $\bn\in  P'(\bN)$, ${\mathbf 1}$ is a vector with coordinates $1$ for all $j$,  we have for $1\le q\le\infty$
$$
\left\|\frac{1}{m}\sum_{\nu=1}^m a_\nu \mathcal V_\bN (\bx - \xi^\nu)\right\|_q \le C(d)(bv(\bN)/m)^{1-1/q} \left(\frac{1}{m}\sum_{\nu=1}^m |a_\nu|^q\right)^{1/q}.
$$
\end{Lemma}

\begin{Theorem}\label{T3.2} Let a set $\Xi_m:=\{\xi^\nu\}_{\nu=1}^m\subset \T^d$ satisfy the condition $|\Xi_m\cap [\bx(\bn),\bx(\bn+{\mathbf 1}))|\le b(d)$, $\bn\in  P'(\bN)$, ${\mathbf 1}$ is a vector with coordinates $1$ for all $j$. Then for $m\ge \vartheta(\bN)$ we have for each $f\in \Tr(\bN)$ and $1\le q\le \infty$
$$
\left(\frac{1}{m}\sum_{\nu=1}^m |f(\xi^\nu)|^q\right)^{1/q} \le C(d)\|f\|_q.  
$$
\end{Theorem}
\begin{proof} For a set $\Xi_m:=\{\xi^\nu\}_{\nu=1}^m\subset \T^d$ denote
$$
M(f):=M(f,\Xi_m) := (f(\xi^1),\dots,f(\xi^m)).
$$
  We have for $f\in \Tr(\bN)$
\begin{align*}
\bigl\| M(f) \bigr\|_q^q &=\frac{1}{m}\sum_{\nu=1}^{m} \bigl| f(\xi^\nu) \bigr|^q =
\frac{1}{m}\sum_{\nu=1}^{m}f(\xi^\nu)\varepsilon_\nu \bigl| f(\xi^\nu) \bigr|^{q-1} =\\
&=(2\pi)^{-d}\int_0^{2\pi}f(\bx)\frac{1}{m}\sum_{\nu=1}^{m}\varepsilon_\nu
\bigl|f(\xi^\nu)\bigr|^{q-1}\mathcal V_\bN(\bx-\xi^\nu)d\bx\le\\
&\le \|f \|_q\left \|\frac{1}{m}\sum_{\nu=1}^{m}\varepsilon_\nu \bigl| f(\xi^\nu)
\bigr|^{q-1}\mathcal V_\bN (\bx - \xi^\nu) \right\|_{q'}.
\end{align*}
Using Lemma \ref{L3.3} we see that the last term is
$$
\le C(d) \| f \|_q \bigl\| M(f) \bigr\|_q^{q-1},
$$
which implies the required bound.

The theorem is proved.
\end{proof}
\subsection{Lower bounds for the discrete norm} \label{Lq2}
We now proceed to the inverse bounds for the discrete norm. Denote
$$
\Delta(\bn):= [\bx(\bn),\bx(\bn+{\mathbf 1})),\quad \bn \in P'(\bN).
$$
Suppose that a sequence $\Xi_m:=\{\xi^\nu\}_{\nu=1}^m\subset \T^d$ has the following 
property. 

{\bf Property} E($b$). There is a number $b\in \N$ such that for any $\bn \in P'(\bN)$ we have
$$
|\Delta(\bn)\cap \Xi_m| =b.
$$
Clearly, in this case $m=v(\bN)b$, where $v(\bN) = |P'(\bN)|$. 

\begin{Lemma}\label{L3.4} Suppose that two sequences $\Xi_m:=\{\xi^\nu\}_{\nu=1}^m\subset \T^d$ and $\Gamma_m:=\{\gamma^\nu\}_{\nu=1}^m\subset \T^d$ satisfy the following condition. For a given $j\in \{1,\dots,d\}$, $\gamma^\nu$ may only differ from 
$\xi^\nu$ in the $j$th coordinate. Moreover, assume that if $\xi^\nu\in \Delta(\bn)$ then 
also $\gamma^\nu\in \Delta(\bn)$. Finally, assume that $\Xi_m$ has property E($b$) with $b\le C'(d)$. 
Then for $f\in \Tr(\bK)$ with $\bK\le \bN$ we have 
$$
\frac{1}{m}\sum_{\nu=1}^m\bigl||f(\xi^\nu)|^q-|f(\gamma^\nu)|^q\bigr| \le C(d,q)(K_j/N_j)\|f\|_q^q.
$$
\end{Lemma}
\begin{proof} Using the simple relation $v^q-u^q = qc^{q-1}(v-u)$ with $c\in [u,v]$ for numbers $0\le u\le v$
we write
$$
\frac{1}{m}\sum_{\nu=1}^m\bigl||f(\xi^\nu)|^q-f(\gamma^\nu)|^q\bigr| \le \frac{q\pi}{2mN_j}\sum_{\nu=1}^m|f'_{x_j}(\alpha^\nu)||f(\beta^\nu)|^{q-1}
$$
with $\alpha^\nu$ and $\beta^\nu$ belonging to the same $\Delta(\bn)$ as $\xi^\nu$ and 
$\gamma^\nu$. Using H{\" o}lder inequality with parameters $q$ and $q'$ we continue the above inequality
$$
\le \frac{q\pi}{2N_j}\left(\frac{1}{m}\sum_{\nu=1}^m|f'_{x_j}(\alpha^\nu)|^q\right)^{1/q}\left(\frac{1}{m}\sum_{\nu=1}^m|f(\beta^\nu)|^q\right)^{(q-1)/q}.
$$
Functions $f'$ and $f$ are from $\Tr(\bK)\subset \Tr(\bN)$ and sequences $\{\alpha^\nu\}_{\nu=1}^m$ and $\{\beta^\nu\}_{\nu=1}^m$ satisfy the conditions of Theorem \ref{T3.2}. Thus, we continue
\be\label{3.6}
\le  \frac{C(d,q)}{2N_j}\|f'_{x_j}\|_q\|f\|_q^{q-1}.
\ee
Using the Bernstein inequality we obtain $\|f'_{x_j}\|_q\le K_j\|f\|_q$. Substituting it into (\ref{3.6}) we complete the proof of Lemma \ref{L3.4}. 

\end{proof}

We now complete the proof of Theorem \ref{T3.1}. As in the proof of Theorem \ref{T2.1} we 
use the $(t,r,d)$-nets with an appropriate $r$. Then Theorem \ref{T3.2} gives us an upper bound for the discrete norm. For the lower bound we use the net $\{\bx(\bn), \bn\in P'(\bN)\}$ 
with an appropriate $\bN$, which satisfies the desired lower bound by (\ref{1.4'}). Then we use Lemma \ref{L3.4} with appropriate $\bK$ to derive from (\ref{1.4'}) the lower bound for 
the set build with a help of the $(t,r,d)$-nets. We now proceed to the detailed argument.

 Let $t$ be such that $t\le t(d)$ and the $(t,r,d)$-net exists. A construction of such nets for all $d$ and $t\ge Cd$, $r\ge t$  is given in \cite{NX}. For a given $n\in \N$ define $r:= n+t+a(d,q)d$ with $a(d,q)\in\N$, which will be specified later. Consider the $(t,r,d)$-net $T(t,r,d)$. Denote $T_\pi(t,r,d):= 2\pi T(t,r,d):=\{2\pi \bx: \bx\in T(t,r,d)\}$. Take an $\bs\in \Z^d_+$ such that $\|\bs\|_1=n$ and consider $\bs'$ with $s_j' := s_j+a(d,q)$. Consider a dyadic box of the form 
 \be\label{3.7}
 [2\pi(a_1-1)2^{-s'_1},2\pi a_12^{-s'_1})\times\cdots\times[2\pi(a_d-1)2^{-s'_d},2\pi a_d2^{-s'_d}).
\ee 
 The volume of this box is equal to
$(2\pi)^d2^{t-r}$. Specify $\bN:= 2^{\bs'-{\mathbf 2}}$ and $\bK:= 2^\bs$. Then each $\Delta(\bn)$, $\bn\in P'(\bN)$, is a dyadic box of the above form (\ref{3.7}). 
Therefore, for each $\bn\in P'(\bN)$, by the definition of the $T_\pi(t,r,d):=\{\xi^\nu\}_{\nu=1}^{2^r}$, there are $2^t$ points $\xi^{\nu}$ from the box $\Delta(\bn)$. We plan to apply Lemma \ref{L3.4} and define the set $\Gamma_m$ in such a way that each $\bx(\bn)$, $\bn\in P'(\bN)$, is included in $\Gamma_m$ exactly $2^t$ times. Then from the definition of $\bs'$ it follows that
$$
K_j/N_j = 2^{2-a(d,q)}, \quad j=1,\dots,d.
$$
Certainly, we can make $a(d,q)$ large enough to get from Lemma \ref{L3.4} the inequality
$$
\frac{1}{m}\sum_{\nu=1}^m\bigl||f(\xi^\nu)|^q-|f(\gamma^\nu)|^q\bigr| \le C_1(d,q)\|f\|_q^q/2
$$ 
where $C_1(d,q)$ is from (\ref{1.4'}). This completes the proof of the lower bound in Theorem \ref{T3.1}.
The upper bound follows directly from Theorem \ref{T3.2}.
 
 \end{proof}
 
 \section{Discussion}\label{D}
 Theorems \ref{T2.1} and \ref{T3.1} give sets of cardinality $\le C(d,q)2^n$, which provide 
 universal discretization in $L_q$, $1\le q\le \infty$, for the collection $\cC(n,d)$. This, basically, solves the Universal discretization problem for the collection $\cC(n,d)$, because, 
 obviously, the lower bound for the cardinality of a set, providing the Marcinkiewicz discretization theorem for $\Tr(R(\bs))$ with $\|\bs\|_1=n$, is $\ge C(d)2^n$. 
 
 In Section \ref{Inf} we discussed the idea of using the set, which provides the Marcinkiewicz discretization theorem for $\Tr(Q_n)$ as the one providing universal discretization for the collection $\cC(n,d)$. We pointed out there that the known results on the Marcinkiewicz discretization theorem for $\Tr(Q_n)$ in $L_\infty$ show that this idea does not work well in case $L_\infty$ -- we loose an extra factor $m^c$, $c>0$. We now discuss an application of this idea in the case of $L_q$, $1\le q<\infty$. We begin with the case $q=2$. The following theorem is obtained in \cite{VT160}.
\begin{Theorem}\label{T4.1} There are three positive absolute constants $C_1$, $C_2$, and $C_3$ with the following properties: For any $d\in \N$ and any $Q\subset \Z^d$   there exists a set of  $m \le C_1|Q| $ points $\xi^j\in \T^d$, $j=1,\dots,m$ such that for any $f\in \Tr(Q)$ 
we have
$$
C_2\|f\|_2^2 \le \frac{1}{m}\sum_{j=1}^m |f(\xi^j)|^2 \le C_3\|f\|_2^2.
$$
\end{Theorem}

In \cite{VT160} we showed how to derive Theorem \ref{T4.1} from the recent paper by  S. Nitzan, A. Olevskii, and A. Ulanovskii \cite{NOU}, which in turn is based on the paper of A. Marcus, D.A. Spielman, and N. Srivastava \cite{MSS}. 
 Theorem \ref{T4.1}, basically, solves the problem of the Marcinkiewicz-type discretization theorem for the $\Tr(Q_n)$ in the $L_2$ case.
The reader can find some more discussion of the $L_2$ case in 
\cite{VT161}. We also refer to the paper \cite{Ka} for a discussion of a
recent outstanding progress in the area of submatrices of orthogonal matrices. 
In particular, in the case of $Q=Q_n$ Theorem \ref{T4.1} provides a set of  $m \le C_1|Q_n| \le C(d)2^nn^{d-1}$ points $\xi^j\in \T^d$, $j=1,\dots,m$, such that for any $f\in \Tr(Q_n)$ 
we have
$$
C_2\|f\|_2^2 \le \frac{1}{m}\sum_{j=1}^m |f(\xi^j)|^2 \le C_3\|f\|_2^2.
$$
Thus, this set with cardinality $\le C(d)2^nn^{d-1}$ is universal for discretization in $L_2$ of the collection $\cC(n,d)$. It shows that on this way we only loose an extra factor $(\log m)^{d-1}$ in the cardinality of the set. It is clear that we cannot do better on this way because dim$(\Tr(Q_n)) \ge C(d)2^nn^{d-1}$. 

Let us discuss the case $L_1$.  The following Marcinkiewicz-type theorem for discretization of the $L_1$ norm of polynomials from $\Tr(Q)$ was proved in \cite{VT161}. 
\begin{Theorem}\label{T4.2} There is large enough constant $C_1(d)$ with the property: For any $Q\subset \Pi(\bN)$ with $\bN=(2^n,\dots,2^n)$
 there exists a set of $m \le C_1(d)|Q|n^{7/2}$ points $\xi^j\in \T^d$, $j=1,\dots,m$ such that for any $f\in \Tr(Q)$ 
we have
$$
\frac{1}{2}\|f\|_1 \le \frac{1}{m}\sum_{j=1}^m |f(\xi^j)| \le \frac{3}{2}\|f\|_1.
$$
\end{Theorem}
This theorem applied to $\Tr(Q_n)$ gives a set of cardinality $m \le C_1(d)|Q_n|n^{7/2}$, which provides universal discretization in $L_1$ for the collection $\cC(n,d)$. Thus, on the way of using the above idea we only loose an extra factor $(\log m)^{d+5/2}$ in the cardinality of the set.

In the case $q\neq 2$ and $q\neq 1$ the known results on the Marcinkiewicz-type theorem 
for the $\Tr(Q_n)$ are weaker than in the case $q=2$ or $q=1$. In this case known results from \cite{VT160} only give us results with an extra power factor $m^{1-1/q}$.

\end{document}